\documentclass{amsart}
\usepackage{amsmath}
\usepackage{amsthm}
\usepackage{amsfonts}
\usepackage{natbib}
\usepackage{amssymb}
\usepackage{enumerate}
\usepackage{graphicx}
\theoremstyle{plain}
\newtheorem{thm}{Theorem}[section]
\newtheorem{prop}[thm]{Proposition}

\newtheorem{cor}[thm]{Corollary}

\newtheorem{no}[thm]{Notation}

\newtheorem{clm}[thm]{Claim}
\newtheorem{lemma}[thm]{Lemma}
\theoremstyle{definition}
\newtheorem{dfn}[thm]{Definition}
\theoremstyle{remark}
\newtheorem*{rmk}{Remark}

\usepackage{MnSymbol}

\usepackage[all]{xy}

\usepackage{xcolor}
\usepackage[normalem]{ulem}
\usepackage{comment}
\usepackage{mathtools}

\usepackage{ushort}
\newcommand{\path}[1]{\mathbf{\ushort{#1}}}

\newcommand{\diam }[1]{{\textbf{diam}\big(#1\big)}}

\newcommand{\ax}{\mathrm{Ax}}

\newcommand{\QT}{\mathcal C_L(\mathbb S)}
\newcommand{\PC}{\mathcal P_K(\mathbb S)}

\usepackage[bookmarks=true, pdfauthor={YANG Wenyuan}]{hyperref}

\thanks{The second author is supported by NSF of China No.11471318, No.11671057 and No.11688101;
the third author is supported by NSF of China No.11771022}

\begin{document}
\title[Large quotients of groups]{Large quotients of group actions with a contracting element}

\author{ Zunwu He, Jinsong Liu}
\address{HLM, Academy of Mathematics and Systems Science,
Chinese Academy of Sciences, Beijing, 100190, China $\&$ School of
Mathematical Sciences, University of Chinese Academy of Sciences,
Beijing, 100049, China} \email{ 519971914@qq.com,
liujsong@math.ac.cn}

\author{Wenyuan Yang}
\address{Beijing International Center for Mathematical Research $\&$ School of Mathematical
Sciences, Peking University, Beijing, 100871, China}
\email{yabziz@gmail.com}
\maketitle
\begin{abstract}
For any proper action of a non-elementary group $G$ on a proper geodesic metric space, we show that if $G$ contains a contracting element, then there exists a sequence of proper quotient groups whose growth rate tends to the growth rate of $G$. Similar statements are obtained for a product of proper actions with contracting elements. 

The tools involved in this paper include the extension lemma for the construction of large tree, the theory of rotating families developed by F. Dahmani, V. Guirardel and D. Osin \cite {2}, and the contruction of a quasi-tree of metric spaces introduced by M. Bestvina, K. Bromberg and K. Fujiwara \cite {1}. Several applications are given to CAT(0) groups and mapping class groups.

\end{abstract}
  \section{Introduction}
Assume that a group $G$ acts isometrically and properly on a proper geodesic metric space $(X,d)$. Let $o$ be a basepoint in $X$. Denote  $N(o,n)=\{g\in G: d(go,o)\leq n\}$ for $n\ge 0$. Define the \textit{growth rate} as follows
$$
\delta_{G}=\limsup\limits_{n\rightarrow \infty} \frac{\log\sharp N(o,n)}{n},
$$
which does not depend on the basepoint $o$. Consider a normal subgroup $N\lhd G$, and the associated  quotient space $(X/N,\bar d)$, where the metric is given by
$$
\bar d(\bar x,\bar y)=\inf\limits_{a,b\in N}d(ax,by)=\inf\limits_{a\in N}d(ax,y).
$$
The quotient group $\bar G=G/N$ acts isometrically and properly on the metric space $(X/N, \bar d)$. For the basepoint $\bar o:=No \in X/N$, one has similarly the \textit{growth rate} for the quotient action:
$$
\delta_{\bar G}=\limsup\limits_{n\rightarrow \infty} \frac{\log\sharp  N(\bar o,n)}{n},
$$
where $N(\bar o, n)=\{\bar g\in \bar G: \bar d(\bar g \bar o, \bar o)\leq n\}$.

Following R. Grigorchuk and P. de la Harpe [11], the action $G\curvearrowright X$ is called \textit{growth tight} if $\delta_{G}>\delta_{G/N}$ for any infinite normal subgroup $N\lhd G$. We remark that if $N$ is finite, it is always true that $\delta_{G}=\delta_{G/N}$. Hence, we are only interested in \textit{proper quotients $G/N$} which means the normal subgroup $N$ is infinite. By abuse of language, we often say that the group $G$ is growth tight if the action is clear in context.

A group is called \textit{non-elementary} if it is not a finite extension of the integer group or a trivial group. It is clear that an elementary group is never growth tight. The following examples of  growth tight actions of non-elementary groups have been established:
\begin{enumerate}
\item Finitely generated free groups with respect to their Cayley graph \cite {6}.
\item Hyperbolic groups \cite {7}.
\item Free product and cocompact Kleinian groups \cite {8}\cite {9}.
\item Geometrically finite Kleinian groups with parabolic gap property \cite {10}.

\item Mapping class group $Mod(S_g)$  of a closed orientable surface $S_g$ of $g\ge 2$ acting on the Teichm\"{u}ller space $T_g$ \cite {11}.

\item CAT(0) groups   with rank-1 elements \cite{11}\cite{Y1}.
\item $l^p$-product of two growth tight actions for $\infty > p\ge 1$ \cite {5}.
\end{enumerate}

Except the last item (7), the third-named author  \cite {4} generalizes the growth tightness of the above class of groups to a more general class of group actions called \textit{statistically convex-cocompact actions} (SCC actions). The definition of a SCC action is irrelavent in this paper: we only point out that it is a generalization of convex-cocompact actions in a probability sense, and includes all examples in the list (1-6).

\begin{thm}
A non-elementary group admitting a SCC action on a proper geodesic space with a contracting element is growth tight.
\end{thm}

In view of growth tightness,  a natural question to ask is whether there is a gap between $\delta_{G}$ and $\sup \{\delta_{G/N}\}$, where the supremum is taken over all proper quotients.

For a torsion free hyperbolic group,  R. Coulon \cite {Coulon} showed that there exists no gap for the natural action on its Cayley graph. In \cite{3}, the third-named author establishes the same result for relatively hyperbolic groups with actions on Cayley graphs and  cusp-uniform actions with parabolic gap property on hyperbolic spaces.

The goal of the present paper is to give a far reaching generalization of the above mentioned results. In fact, our main result exhibits a sequence of proper quotient groups with no gap for any proper action of a group with a contracting element.

\begin{thm}\label{mainthm}
Suppose that a non-elementary group $G$ acts properly   on a proper geodesic metric space $X$ with a contracting element. Then there exists a sequence of proper quotient groups $\bar{G}_n$ such that, as $n \rightarrow\infty$,
\begin{equation}
\delta_{\bar{G}_n}\rightarrow\delta_{G}.
\end{equation}
\end{thm}

\begin{rmk}
When the action $G\curvearrowright X$ is SCC, the action is growth tight by Theorem 1.1, and thus $\delta_{\bar{G}}<\delta_{G}$ for every proper quotient $\bar G$. Hence, this is the setting where Theorem 1.2 is interesting.

On the other hand, if the assumption of SCC actions is dropped, there exists indeed a sequence of proper quotient groups such that $\delta_{\bar{G}_n}=\delta_{G}$. For example, if the group $G$ contains a \textit{large} parabolic subgroup $P$ with the same growth rate of $G$, then we can construct a sequence of proper quotient groups with the same growth rate with that of $G.$ In this sense, our theorem is sharp. For details, we refer the reader to the proof of \cite[Proposition 1.5]{3} 
\end{rmk}

We now consider the applications of our result to mapping class groups and CAT(0) groups with rank-1 elements.

In \cite {13}, Y. Minsky proved that any pseudo-Anosov element in the mapping class group is contracting with respect to the Teichm\"{u}ller metric. Thus, our theorem applies and gives the following corollary:
\begin{cor}
Let $S=S_g$ be a closed orientable surface of genus $g\ge 2$ and $G$ be the mapping class group of $S$.
Denote by $T_{g}$ its Teichm\"{u}ller space endowed with Teichm\"{u}ller metric. Then there
exists a sequence of proper quotient groups $\bar{G}_m$ of $G$ such that
$\delta_{\bar{G}_m}\rightarrow\delta_{G}$ as $m \rightarrow\infty$.
\end{cor}

\begin{rmk}
The growth rate of the mapping class group $G \curvearrowright T_{g}$ with the Teichm\"{u}ller metric  equals $6g-6$ \cite [Theorem 1.3]{16}. We remark that the growth tightness of the action $G \curvearrowright T_{g}$ was proved in \cite{11}.
\end{rmk}

Notice that a proper cocompact action $G\curvearrowright X$  (i.e. there exists a compact subset $K$ in $X$ such that $G\cdot K=X$) is called \textit{geometric}. A \textit{CAT$(0)$ group} $G$ means a geometric action on a proper CAT$(0)$ space $X$. 

A hyperbolic isometry of a proper CAT$(0)$ space is \textit{rank-one} if it has an axis that doesn't bound a half Euclidean plane. A well known fact is that a rank-one element is equivalent to a contracting element, provided that the space is proper CAT$(0)$ space \cite [Theorem 5.4] {15}. Thus, the following corollary of our result is immediate.

\begin{cor}\label{CAT0NoGap}
Assume that a non-elementary group $G$ acts geometrically on a proper CAT$(0)$ space $X$ with a rank-1 element. Then there exists a
sequence of proper quotient groups $\bar{G}_m$ such that
$$\delta_{\bar{G}_m}\rightarrow\delta_{G}, \:when \: m
\rightarrow\infty.$$
\end{cor}

We now turn to our second main result of this paper in  a product of geometric actions. This is motivated by the work  of C. Cashan and J. Tao \cite{8} who proved the growth tightness for products of geometric actions with contracting elements. In particular, they produced the first example of groups (e.g. $\mathbb F_2\times \mathbb F_2$) so that it is growth tight for one generating set, but not necessarily for the others. Before stating our general theorem, let's look at a particular class of CAT(0) groups.

When $X$ is a CAT$(0)$ cube complex, we often say $G$ is a cubical group. In this case, we are able to remove the assumption of existence of rank-1 elements in \ref{CAT0NoGap}. This makes use of a remarkable result of M. Sageev and P. Caprace \cite{12} which resolves the rank rigidity conjecture for cubical groups.

In  \cite [Theorem A]{12}, it is  showed that if a CAT$(0)$ cube complex $X$ can not be factored as a product of CAT$(0)$ cube complexes, then any geometric group action on $X$  contains a rank-one element. Using this result, one can derive that any geometric action of $G$ on a CAT(0) cube complex virtually splits as a product action on a product of irreducible cube complexes such that each action on each irreducible factor contains a rank-1 element. By  \cite[Theorem 1.2]{Y1} we see that any geometric action with rank 1 elements on a CAT(0) cube complex is growth tight with respect to the CAT$(0)$ metric.

\begin{cor}\label{CubeNoGap}
Assume a group $G$ acts  by a geometric action on a proper geodesically complete CAT$(0)$ cube complex
$X$. Then there exists a
sequence of proper quotient groups $\bar{G}_m$ such that
$$\delta_{\bar{G}_m}\rightarrow\delta_{G}, \:when \: m
\rightarrow\infty.$$
\end{cor}

An important class of cubical groups is provided by the class of Right Angled-Artin group (RAAG). Given a finite simplicial graph $\Gamma$, a right-angled Artin group $G$ associated with the graph is given by the following  group presentation:
$$G=\langle V(\Gamma)|v_1v_2=v_2v_1 \text{ iff }  (v_1,v_2)\in E(\Gamma)\rangle.$$

Assume the graph $\Gamma$ has $k$ vertices. Let $T$ be the torus of dimension $k$ with edges labelled by the vertices. The subcomplex $Y$ of $T$ consists of all the faces whose edge labels span a complete subgraph of $\Gamma$. The resulting space $Y$ is called the \textit{Salvetti complex} of $G$. It is well known that its universal covering space is a CAT$(0)$ cube complex. Thus, the following   consequence of the above corollary is direct.


\begin{cor}\label{RAAGNoGap}
Let $G$ be the right-angled Artin group acting on the universal covering of the Salvetti complex associated to the graph $\Gamma$. Then there exists a sequence of proper quotient groups $\bar{G}_m$ such that, as $m \rightarrow \infty$,
$$
\delta_{\bar{G}_m}\rightarrow\delta_{G}.
$$
\end{cor}
\begin{rmk}
In \cite{5} C. Cashan and J. Tao proved that the action of $\mathbb F_2\times \mathbb F_2$ on its universal covering of the Salvetti complex is growth tight with respect to the CAT$(0)$ metric. It is worth remarking that the action of the same group is not growth tight, when the combinatorial metric is used (i.e. the $L_1$-metric induced by the 1-skeleton of the universal cubical covering).
\end{rmk}

The proof of Corollary \ref{CubeNoGap} relies on the following result on product of groups actions, which is of independent interest.

Let $G=\prod_{1\le i\le n} G_i$ be a product of groups. We call a quotient $\bar G$ of $G$ \textit{proper} if the kernel of $G\to \bar G$ projects to an infinite normal subgroup  in each factor $G_i$.   

\begin{thm}\label{productThm}
Let $G=\prod_{1\le i\le n} G_i$ be a product of groups, and
$X=\prod_{1\le i\le n} (X_i,d_i)$   a product
of geodesic metric spaces such that for  each $1\leq
i\leq n$, the action $G_i\curvearrowright X_i$ is a geometric action with a contracting element. We equip
$X$ with the $L^p$-metric, where $1\leq p\leq \infty$. Then
there exists a sequence of proper quotient groups $\bar{G}_m$ such that
$$\delta_{\bar{G}_m}\rightarrow\delta_{G}, \: when\:\: m
\rightarrow\infty.$$
\end{thm}
\begin{rmk}
When $1\le m <\infty$, the product action of $G$ on $X$ is proved to be growth tight in \cite{5}. Thus, our result implies that there exists no gap between the quotient growth rates and the whole growth rate in this case. 
\end{rmk}


To conclude the introduction, we outline the proof of    Theorem \ref{mainthm} and describe the structure of the paper. The section \ref{Section2} introduces the preliminary materials and the final section \ref{Section6} gives the proof of Theorem \ref{productThm}.  The proof of Theorem \ref{mainthm}  follows the same strategy as in \cite{3} and can be subdivided into the following three steps contained respectively in Sections  \ref{Section3},  \ref{Section4} and  \ref{Section5}:
\begin{enumerate}
\item Construct a ``large'' tree $\mathcal T$ labeled by a free semigroup $\Gamma$ in $G$ in the sense that its growth rate  approximates that of $G$. Moreover, this is a rooted quasi-geodesic tree where each path is admissible. This is done by choosing a particular contracting element $f$ using the extension lemma. The dertailed statement is given in Lemma \ref{freesemigroup} in Section \ref{Section3}. 

\item Choose another  contracting element $h$, which is independent from $f$. Let $\mathbb S=\{g\ax(h): g\in G\}$. Following M. Bestvina, K. Bromberg and K. Fujiwara \cite{1}, we then construct the projection complex $\mathcal P_K(\mathbb S)$ and the quasi-tree of metric spaces $\mathcal C_L(\mathbb S)$. The space $\mathcal C_L(\mathbb S)$ is a quasi-tree, hence it is hyperbolic. The crucial point of these constructions is that it allows us to lift the geodesic from  $\mathcal C_L(\mathbb S)$ to $X$, and meanwhile keep tracking to the closeness with certain contracting subsets in $\mathbb S$.

\item
In the last step, we apply the theory of rotating family in F. Dahmani, V. Guirardell and D. Osin \cite{2} to analyze the features of elements $g$ in the kernel $\llangle h^n\rrangle$ for $n\gg 0$. This is achieved through the action of $G$ on the hyperbolic cone-off construction to the 
quasi-tree space $\mathcal C_L(\mathbb S)$ along $\mathbb S$.  

Via the step (2) and results in \cite{3}, we shall prove that the constructed quasi-geodesic tree $\mathcal T$ is sent to $\bar X=X/\llangle h^n\rrangle$ injectively. This completes the proof of  Theorem \ref{mainthm} in Section \ref{SSection52}.
\end{enumerate}

\section*{Acknowledgement}
The authors are grateful  to  the anonymous referee for her/his many comments which improves  the presentation and the readability of the article.

\section{Preliminaries}\label{Section2}
\begin{no}
\end{no}

\begin{enumerate}
\item $(X,d)$: a proper geodesic metric space.
\item $N_c(A)$: the $c$-neighborhood of a subset $A$ in $X$.
\item $\pi_A(B)$: the set of closest projection points from $B$ to $A$.  Denote by $d^{\pi}_A(B)$ the diameter of $\pi_A(B)$ and $d^{\pi}_A(B,C)$ the diameter of $\pi_A(B\cup C)$.
\item $d_{GH}(,)$: the Gromov-Hausdorff distance.
\item $[x,y]$: a geodesic segment from $x$ to $y$. $[x,y]_-=x$, $[x,y]_+=y$.
\item $[u,v]_\gamma$: the sub-path of the path $\gamma$ from $u$ to $v$.
\item $a\gtrsim b$: $b-a$ is bounded above by a universal constant. $a\thicksim b$ means $a\gtrsim b$ and $b\gtrsim a$.
\end{enumerate}

A map $f:(X_1,d_1)\longrightarrow (X_2,d_2)$ is called a
\textit{$(\lambda,c)-$quasi-isometric map} for  $\lambda\geq 1$,
$c>0$ if for any $x,y\in X_1$, we have
$$1/\lambda\cdot d_1(x,y)-c\leq d_2(f(x), f(y))\leq\lambda\cdot  d_1(x,y)+c.$$
In addition, if  $X_2$ stays within a finite neighborhood of $f(X_1)$, then $f$ is called a \textit{$(\lambda,c)-$quasi-isometry}, and we say that $X_1$ and $X_2$ are \textit{quasi-isometric}.

When $(X_1,d_1)$ is $[0,L]$ (resp. $[0,\infty)$ or $\mathbb{R}$), we say that $f$ is  a \textit{$(\lambda,c)-$quasi-geodesic (resp. quasi-geodesic ray, or bi-infinite quasi-geodesic)}.

A metric space is called a \textit{quasi-tree} if it is quasi-isometric to a tree.

 \subsection{Contracting property and admisible paths}
We introduce the basic notion of a contracting element and then recall several  related results from \cite{Y1} and \cite{4}. 
\begin{dfn}
Assume the group $G$ acts properly on the space $(X,d)$ and choose a basepoint $o\in X$. 
\begin{enumerate}
\item
A subset $S\subset (X,d)$ is called \textit{$C$-contracting} for a constant $C\ge 0$, if for every geodesic $\alpha$ in $X$ with $d(\alpha,S)\geq C$, we have $d^{\pi}_S(\alpha)\leq C$.
\item
An element $h\in G$ is called \textit{contracting} if  the map $n\in \mathbb Z \mapsto h^n o$ is a $(\lambda,c)-$bi-infinite quasi-geodesic for some $\lambda\geq 1$ and $c>0$, and  $\langle h\rangle o$ is $C$-contracting for some $C\ge 0$.
\end{enumerate}
\end{dfn}

\begin{rmk}
It is an easy exercise that the contracting property of an element does not depend on the choice of the basepoint.
\end{rmk}

A collection of $C$-contracting subsets  $\mathbb S$ in $(X, d)$ with a uniform contracting constant $C>0$ is called a \textit{$C$-contracting system} (or simply, a contracting system if $C$ is  clear in the conext).

By definition, two subsets $A$ and $B$ have   \textit{$\mathcal R$-bounded intersection}   for a function $\mathcal R: [0, \infty) \to [0, \infty)$ if  we have
    $$\diam{N_{c}(A)\bigcap N_{c}(B)}\leq \mathcal R(c),$$
for any $c>0$. A collection $\mathbb S$ of subsets has \textit{bounded $\mathcal R$-bounded intersection} if any two distinct $A, B\in \mathbb S$ have $\mathcal R$-bounded intersection. 
\begin{lemma}\label{bpEQbi}
A contracting system $\mathbb S$ has   $\mathcal R$-bounded intersection if and only if  $\mathbb S$ has bounded projection: there exists $D>0$ such that $d^{\pi}_A(B)\le D$ for any $A\ne B\in \mathbb S$.
\end{lemma}

We now list some useful properties about the contracting subsets.
\begin{lemma}\label{pp}
Let $A\subset (X,d)$ be $C$-contracting. The following holds:
\begin{enumerate}
\item (quasi-convexity) If $\gamma$ is a geodesic with endpoints in $A$, then $\gamma\subset N_{3C}(A)$.
\item (Lipschitz projection) If $\gamma$ is a geodesic in $X$, then $d^{\pi}_A(\gamma)\leq Length(\gamma)+C_1(C)=Length(\gamma)+C_1.$
\item  If $[a,b]$ is a geodesic in $X$, then $d^{\pi}_A([a,b]\lesssim d^{\pi}_A(\{a,b\}).$
\item Let $\alpha$ be a geodesic such that $d_A^{\pi}(\alpha_-, \alpha_+) >K>2C$. Then $$\diam{N_C(A)\cap \alpha}\ge K/2.$$ 
\end{enumerate}
\end{lemma}
\begin{proof}
The first two
and last properties are easy exercises and left to the interested reader, and here we only give a sketch of the proof of the third one.

Let $c$ and $d$ be the first enter point and last exit point of $[a,b]$ with respect to $A$. If we show the length of the subpath $[c,d]\subset[a,b]$ is bounded by $d^{\pi}_A(\{a,b\})$ and a constant depending only on $A$, by the second property and contracting property, we can finish the proof.

On the one hand, $d^{\pi}_A[a,c]$ and $d^{\pi}_A[d,b]$ are both small, where $[a,c]\subset [a,b]$ and $[d,b]\subset [a,b]$ . By the triangle inequality, $d^{\pi}_A(\{c,d\})$ is bounded by $d^{\pi}_A(\{a,b\})$ and a constant depending only on $A$. On the other hand, both $d(c,A)$ and $d(d,A)$ are small. Thus $d(c,d)=Length[c,d]$ is bounded by $d^{\pi}_A(\{a,b\})$ and a constant depending only on $A$.
\end{proof}

Let $A$ be a subset of elements in $G$. We usually write $A\cdot o$ or even simpler $Ao$ as the set of orbital points of $A$ on the basepoint $o$, i.e. $Ao=\{a\cdot o: a\in A\}$ .

For any contracting element $h\in G$, define  $$E(h)=\{g\in G \mid d_{GH}(\langle h\rangle o,g\langle h\rangle o)<\infty\},$$ where $d_{GH}$ denotes the Gromov-Hausdorff distance. The subset $\ax(h)=E(h)\cdot o$ is called the \textit{axis} of   $h$.

\begin{lemma}\label{fi}\cite[Lemma 2.11]{4}
Let $h$ be a contracting element. Then $\langle h\rangle$ is of finite index in $E(h)$ and $E(h)=\{g\in G\mid \exists n\in \mathbb{Z}\setminus 0,gh^ng^{-1}=h^{\pm n}\}.$ Moreover, the collection $\{g\ax(h): g\in G\}$ of axes is  a contracting system with bounded intersection.
\end{lemma}

Note that if $g\in E(h)$ and $g$ is of infinite order, then $E(g)=E(h)$.  Two contracting elements $h, k$ are called \textit{independent} if the collection  $\{g\ax(h), g\ax(k): g\in G\}$ has bounded intersection.

\subsection{Admissible paths}
We now introduce the notion of admissible paths relative to a contracting system $\mathbb S$ with $\mathcal R$-bounded intersection. The system $\mathbb S$ will be the set of all $G$-translates of axis of a contracting element (cf. Lemma \ref{fi}). 
\begin{dfn}[Admissible Path]\label{AdmDef} 
Given constants $K, \theta \ge 0$, a path $\gamma$  is called \textit{$(K,
\theta)$-admissible} in  $X$, if  the path $\gamma$ {is piecewise-geodesic and} contains an {alternating} sequence of  disjoint geodesic subpaths $p_i$ $(0\le i\le n)$ in this order, each associated with a contracting subset $S_i \in \mathbb S$, with the following   called \textit{Long Local} and \textit{Bounded Projection} properties:
\begin{enumerate}

\item[\textbf{(LL)}]
Each $p_i$ has length bigger than  $K$, except {when} $(p_i)_- =\gamma_-$ or $(p_i)_+=\gamma_+$,

\item[\textbf{(BP)}]
For each $S_i$,  we have  
$$
d^{\pi}_{S_i}((p_{i})_+,(p_{i+1})_-)\le \theta
$$
and 
$$
d^{\pi}_{S_i}((p_{i-1})_+, (p_{i})_-)\le \theta
$$
when $(p_{-1})_+:=\gamma_-$ and $(p_{n+1})_-:=\gamma_+$ by convention.


\end{enumerate}
\paragraph{\textbf{Saturation}} The collection of $S_i \in \mathbb S$ indexed as above, denoted by $\mathbb S(\gamma)$, will be referred to as contracting subsets for $\gamma$. The union of all $S_i \in \mathbb S(\gamma)$ is called the \textit{saturation} of $\gamma$. 
\end{dfn}

The following property about $(K,\tau)$-admissible path will be  used in the sequel.
\begin{prop}\label{admis}\cite[Proposition 2.7]{4}
For any $\tau>0$, there are constants $B=B(\tau)$, $K=K(\tau)>0$ such that for
any $(K,\tau)$-admissible path $\gamma$ and for any contracting subset $X_i$ associated to the subpath $p_i(1\leq i\leq n)$, we have:
$$d^\pi_{X_i}(\gamma_1)\leq B,d^\pi_{X_i}(\gamma_2)\leq B,$$
where $\gamma_1=[\gamma_-,(p_i)_-]_{\gamma},$  $\gamma_2=[(p_i)_+,\gamma_+]_{\gamma}.$
\end{prop}

 The following notion of transitonal/deep points will be useful in proving Lemma \ref{lastlem}.
 
\begin{dfn}
Fix a contracting system $\mathbb S$ with bounded intersection. 
We say that a geodesic $\gamma$ in $X$ contains \textit{a $(\epsilon,M)$-deep point} relative to $\mathbb S$ if there exists a connected subsegment of $\gamma$ with  length at least $2M$ contained in the $\epsilon$-neighborhood
of $S$ for some $S\in \mathbb S$. 

If $\gamma$ does not contain any  $(\epsilon,M)$-deep point, then it is called \textit{$(\epsilon, M)$-transitional} relative to $\mathbb S$.
\end{dfn}

\section{Construct  large trees in $G$}\label{Section3}
This section presents the first step in  the proof of Theorem \ref{mainthm}. Namely, we construct a quasi-geodesic tree $T$ rooted at a basepoint $o\in X$ so that the growth rate $\delta_T$ can be arbitrarily close to $\delta_G$. This is essentially done in \cite[Section 3]{4}. We now briefly review the construction, which replies on the following extension lemma.

\begin{lemma}\label{extend3}
There exists a set $F=\{f_1, f_2, f_3\}$ of three contracting elements and a constant $\tau>0$ such that for any $g, h\in G$ there exists $f\in F$ with the following property
$$
\max\{d^{\pi}_{\ax(f)}([o, go]), d^{\pi}_{\ax(f)}([o, ho])\}\le \tau.
$$
\end{lemma}

Let  $\mathbb W(A)$ be the set of   all finite words over an alphabet set  $A$.  Consider an evaluation map $\iota: A \to G$, so we can define an \textit{extension map} $\Phi: \mathbb W(A) \to G$ as follows: for any word $W=a_1a_2\cdots a_n \in \mathbb W(A)$, set $$\Phi(W)=\iota(a_1) \cdot f_1\cdot\iota(a_2)\cdot f_2\cdot \cdots \cdot \iota(a_{n-1})\cdot f_{n-1} \cdot\iota(a_n) \in G,$$ where $f_i\in F$ is supplied by the extension lemma \ref{extend3} for each pair $(a_i, a_{i+1})$. 

Denote  $K=\max\{d(o, f_io): i=1, 2, 3\}$, we see that $\Phi(W)$ produces a $(K, \tau)$-admissible path as follows
\begin{equation}\label{AdmisEQ}
\gamma=\path{\iota(a_1)} \cdot\path{f_1}\cdot\path{\iota(a_2)}\cdot \path{f_2}\cdot \cdots \cdot \path{\iota(a_{n-1})}\cdot \path{f_{n-1}} \cdot\path{\iota(a_n)}
\end{equation}
where $\path{\iota(a_i)}$ denotes a geodesic $\iota(a_1)f_1\cdots \iota(a_{i-1})f_{n-1}\cdot [o, a_io]$ and $\path{f_i}$   a geodesic $\iota(a_1)f_1\cdots \iota(a_{i-1})\cdot [o, f_io]$.

Given $n, \Delta>0,$  define the annulus set: 
$$
A(o, n, \Delta) =\{g\in G: |d(o, go)-n|\le \Delta\}.
$$

The main step in constructing a quasi-geodesic tree is summarized in the following statement:

There exist  constants $C, \tau>0$ such that for each $n\gg 0$, there exists a $C$-separated set $A$ from $A(n, \Delta)$ with the following property:      there is a common $f\in F$ so that for  each pair $(a, a') \in A\times A$, we have
$$
\max\{d^{\pi}_{\ax(f)}([o, ao]), d^{\pi}_{\ax(f)}([o, a'o])\}\le \theta.
$$
 Consider the   path (\ref{AdmisEQ}) for $f_i=f$, which is labeled  by a word over the   alphabet set $Af$. By Proposition \ref{admis}, it is a $(K, \theta)$-admisisble path, so it is a $(\lambda, c)$-quasi-geodesic. Furthermore, the sufficiently large  $C$-separation implies that the extension map as above is injective. Hence, $Af$ generates a free semi-group $\Gamma$. 
 
 One can define the Cayley graph of a semi-group relative to a generating set in the same way  in a group. When $Af$ is the free base of $\Gamma$, the \textit{standard Cayley graph} $\mathcal T$ is a tree rooted at identity where each vertex admits $\sharp Af$ edges. This tree structure allows one to compute  the growth rate $\delta_\Gamma$. Indeed, for  any $0<\delta<\delta_G$, there exists a large enough integer $n$ such that the growth rate of $\Gamma o$   is greater than $\delta$.   See the details in \cite[Section 3]{4}.

  We summarize the above discussion into the following result.

\begin{lemma}\label{freesemigroup}
There exists a contracting element $f\in G$ and $K=K(f), \theta=\theta(f)>0$ with the following property.
\begin{enumerate}
\item
For any $0<\delta<\delta_G$, there exists a free-semigroup $\Gamma$ and a quasi-isometric map from the standard Cayley graph $\mathcal T$ of $\Gamma$ to $X$ defined by 
$$
g\in \mathcal T \mapsto go\in X 
$$
such that each geodesic branch  in $\mathcal T$ is sent to a   $(K, \theta)$-admissible  path $\gamma$ relative to the contracting system $\{g\ax(f): g\in G\}$.  
\item
Moreover,  $\delta_\Gamma>\delta$ and $\sharp \Gamma\cap N(o, r)\asymp \exp(\delta_\Gamma r).$

\item
Let $h$ be a contracting element which is independent from $f$. There exists $\epsilon=\epsilon(h,\delta), M=M(h, \delta)>0$ such that any geodesic $\alpha$ with two endpoints in $\gamma$ is $(\epsilon, M)$-transitonal relative to $\mathbb S=\{g\ax(h): g\in G\}$.
\end{enumerate}
\end{lemma}

\begin{rmk}
Although  we have $M\to \infty$ as  $\delta\to \delta_G$, the third statement says roughly the points in $\gamma$ are uniformly deep in $\mathbb S$. This is the key feature to be used in the proof of Lemma \ref{lastlem}.  
\end{rmk}
\begin{proof}[Sketch of the proof]
The first statement (1) with  $\delta_\Gamma>\delta$ was proved in  \cite[Section 3]{4} as briefly recalled above. The first statement (3) is a direct consequence of (1) and Proposition \ref{admis}. Only the purely exponential growth of $\sharp \Gamma\cap N(o, r)\asymp \exp(\delta_G r)$ was not explicitly stated in \cite[Section 3]{4}, but which follows by standard arguments with the image $\Gamma o$ being a contracting subset stated there. We leave its proof to the interested reader.  
\end{proof}

\section{Projection compex and quasi-tree of spaces}\label{Section4}
Following Bestvina-Bromberg-Fujiwara, we present a construction of a projection complex and quasi-tree of spaces, and refer the reader to  \cite {1} for details.

\subsection{Projection complex}
The construction starts with  a collection of subsets satisfying the so-called projection complex axioms given below, and outputs a hyperbolic graph called projection complex. In our specific setting, a proper action with a contracting element always furnishes such a collection. 

Recall that $d^{\pi}_S(U, V)=\diam{\pi_S(U)\cup \pi_S(V)}$.
By Lemma \ref{fi}, we obtain the following result.
\begin{lemma}\label{bp}
Let $h$ be any contracting element. The collection  $\mathbb S=\{g\ax(h): g\in G\}$ is a contracting system  with bounded intersection such that the following projection complex axioms in \cite {1} hold, i.e there exists a constant $\theta=\theta(h,X)>0$ such that
\begin{enumerate}
\item[\textbf{(PC 0)}] ${d^{\pi}_{U}}(V)\leq \theta$ for any $U\ne V\in \mathbb S$.
\item[\textbf{(PC 1)}] $d^\pi_S(U,V)=d^\pi_S(V,U)$.
\item[\textbf{(PC 2)}]
   $d^\pi_S(U,V)+d^\pi_S(V, W)\geq d^\pi_S(U, W)$ for any $S, U, V, W \in \mathbb S$.

\item[\textbf{(PC 3)}]  $\min\{d^{\pi}_{U}(V, W), d^{\pi}_V(U, W)\}\le \theta$ for any distinct triple $(U, V, W)$ in $\mathbb S$.

\item[\textbf{(PC 4)}] $\sharp\{S\in S\mid{d^{\pi}_{S}}(U, V)>\theta\}<\infty$ for any fixed $U, V\in \mathbb S$.
\end{enumerate} 
\end{lemma}
\begin{rmk}
By Lemma \ref{bpEQbi}, the axiom (\textbf{PC0}) follows from the bounded intersection of $\mathbb S$. The axioms (\textbf{PC1}, \textbf{PC2}) are trivial by definition of $d^{\pi}_S(U, V).$ The last two need some work, but was proved in \cite{Sisto} (see also \cite{Y1}).
\end{rmk}
 
In the sequel, we fix the constant $\theta=\theta(h,X)$. Later on, we actually take the   element $h$ to be independent from the contracting element $f$ in Lemma \ref{freesemigroup}.
 
In \cite{1}, a modified distance-like function  is introduced for each $S\in \mathbb S$: $$d_S: \mathbb S\setminus S\times \mathbb S\setminus S\to [0, \infty)$$  so that it is symmetric and agrees with $d^{\pi}_S$ up to a uniform amount $2\theta$. As a consequence, the properties (\textbf{PC3}, \textbf{PC4}) are still true, and the triangle inequality in (PC2) holds up to a uniform error.

We consider the interval set for $K>0$ and $U, V\in \mathbb S$ as follows $$\mathbb S_K(U,V)\triangleq\{S\in \mathbb S\mid d_S(U,V)>K\}.$$
It possesses a total order which is very important in \cite{1} and in what follows.

\begin{lemma}\label{pc}\cite [Theorem 3.3]{1}
There exists a constant $\Theta=\Theta(\theta)>0$ for the above $\mathbb S$  such that  the set $\mathbb S_\Theta(U,V)\bigcup\{U,V\}$ is totally ordered with least element $U$ and great element $V$, such that given $S_0$, $S_1$, $S_2\in \mathbb S_\Theta(U,V)\bigcup\{U,V\}$, if $S_0<S_1<S_2$, then
    $$d_{S_1}(U,V)\lesssim_\theta d_{S_1}(S_0,S_2)\leq d_{S_1}(U,V),$$  and $$d_{S_0}(S_1,S_2)\sim_\theta 0 \quad and \quad d_{S_2}(S_0,S_1)\sim_\theta 0.$$

\end{lemma}


We now give the definition of a projection complex.

\begin{dfn}
The \textit{projection complex} $\PC$ for $K\ge \Theta$ is  a graph with the vertex set consisting of the elements in $\mathbb S$. Two  vertices $u$ and $v$ are connected if $\mathbb S_K(u,v)=\varnothing$. We equip $\PC$ with a length metric $d_{\mathcal P}$ induced by assigning unit length to each edge.
\end{dfn}

The structural  result about the projection complex is the following.
\begin{thm}\cite{1}\label{projcplxThm}
For $K\gg 0$, the projection complex  $\PC$ is a hyperbolic space of infinite diameter, on which $G$ acts co-boundedly.
\end{thm}

 It is proved in \cite[Proposition 3.7]{1} that this interval set $\mathbb S_K(U, V)$ naturally gives rise to a connected path between $U$ and $V$ in $\PC$: the consecutative elements directed by the total order  are actually adjacent in the projection complex.  
 
Moreover, we note the folowing connection  with admissible paths.
\begin{lemma}\label{admisIntvLem}
For any $K>0$, there exists a constant $\tilde K\ge 0$ with $\tilde K\to \infty$ as $K\to\infty$ such that for any two points $u\in U, v\in V$ there exists a $(\tilde K, \theta)$-admissible path $\gamma$ between $u$ and $v$ with saturation $\mathbb S_K(U, V)$. The set $\mathbb S_K(U, V)$ is possibly empty, and in this case, $\gamma$ is a geodesic $[u, v]$. 
\end{lemma}
\begin{proof}
List $\mathbb S_K(U, V)=\{S_1, S_2, \cdots, S_k\}$ by the total order given in Lemma \ref{pc}. Thus,     $$d_{S_i}(S_{i-1}, S_{i+1})\gtrsim_\theta d_{S_i}(U, V)\ge K$$  for any $i$. 

We now build the admissible path as follows. Choose a sequence of points $x_i\in \pi_{S_i}(S_{i-1}), y_i\in \pi_{S_i}(S_{i+1})$ for $1\le i\le k$, where $S_0:=U, S_{k+1}:=V$.
We connect  the points in $\{x, x_1, y_1, \cdots, x_k, y_k, y\}$ consecutatively to get a path $\gamma$. By construction, we have by Lemma \ref{bp} that $$d^{\pi}_{S_i}((x_{i-1})_+,(x_{i})_-), d^{\pi}_{S_i}((y_{i})_+,(y_{i+1})_-) \le \theta,$$ which verifies the condition (\textbf{BP}).  Moreover, since $d_{S_i}(S_{i-1}, S_{i+1}) \sim_\theta d^{\pi}_{S_i}(S_{i-1}, S_{i+1})$, there exists $\tilde K=\tilde K(K)$ such that $d(x_i, y_i)\ge \tilde K$. Thus, $\gamma$ is a $(\tilde K, \theta)$-admissible path relative to $\mathbb S_K(U, V)$. The proof is complete.
\end{proof}

In practice, we always assume that $\tilde K$ satisfies Proposition \ref{admis} so that the path $\gamma$ shall be a quasi-geodesic.

\subsection{Quasi-tree of spaces}

Let $L\ge K$ be a positive number where $K$ is given by Theorem \ref{projcplxThm}. We now define a blowup version, $\QT$, of the projection complex $\PC$  by remembering the geometry of each $S\in \mathbb S$.  

Roughly speaking, we replace each $S\in \mathbb S$, a vertex in $\mathcal P_K(\mathbb S)$, by the corresponding subspace $S\subset X$. Meanwhile, we keep the adjacency relation in $\PC$: if $U$, $V$ are adjacent in $\mathcal C_K(\mathbb S)$ (i.e.  $d_{\mathcal P}(U,V)=1$), then for every point $u$ in $\pi_UV$ and $v$ in $\pi_VU$, we attach an edge of length $L$ connecting $u$ and $v$. Note that $L$ is comparable to $K$, i.e $1/2K\leq L\leq 2K$ by   \cite [Lemma 4.2]{1}.

Since $h$ is contracting, by Lemma \ref{fi}, the infinite cyclic subgroup $\langle h\rangle$ is of finite index in $E(h)$, so  $\ax(h)=E(h)o$ is quasi-isometric to a line $\mathbb{R}$. Thus, $\mathbb S$ consists of  uniformly quasi-lines. By \cite [Theorem B]{1}, we have the following.

\begin{thm}\cite{1}\label{quasitreeThm}
For $L\gg K$, the quasi-tree of spaces  $\QT$ is a quasi-tree of infinite diameter, with each $S\in \mathbb S$ totally geodesically embedded into $\QT$.  Moreover, the shortest projection from $U$ to $V$ in $\QT$ agrees with the projection $\pi_UV$ up to a uniform finite Hausdorff distance.
\end{thm}

By the ``moreover'' statement, we see that $\mathbb S$ is a collection of totally geodesic subsets with bounded intersection. In a hyperbolic space, a quasi-convex subset is exactly contracting in our sense.  Note that an isometry $g$ on a hyperbolic space $\QT$ is called \textit{hyperbolic} if the oribital map $n\to g^no$ is a quasi-isometric embedding map for any $o\in \QT$. Thus, we obtain the following corollary, which will be useful in the next section.
\begin{lemma}\label{RotonQT}
The group $G$ acts co-boundedly on $\QT$ and leaves $G$-invariant the contracting system $\mathbb S$ with bounded intersection. The element $h$ acts as hyperbolic isometry on $\QT$.
\end{lemma}

From the construction, there exists a natural map which collapses each totally  geodesically embedded subspace $S\subset \QT$  as a point and identifies edges of length $L$ to  edges of unit length in $\PC$.  In a reverse direction, the connected path in $\PC$ given by $\mathbb S_K(X, Y)$ lifts to a standard path in  $\QT$ described as follows.   The notion of standard paths plays a key role to establish Theorem \ref{quasitreeThm}.

\begin{dfn}\label{stdpathDef}
A path $\gamma$ from $u\in U$ to $v\in V$ in the space $\QT$ is called \textit{a $K$-standard path} if it passes through the set of vertices in $\mathbb S_K(U,V)\cup\{U, V\}$ in a natural order given in Lemma \ref{pc} and, within each vertex, the path is a geodesic.
\end{dfn}

  The following result is a useful property which we shall use soon.

\begin{lemma}\cite[Lemma 4.15]{1}\label{unicloseLem}
There exist uniform constants $K', D>0$  such that any geodesic $\alpha$ from $u\in U$ to $v\in V$ in $\QT$ passes through each $S\in \mathbb S_{K'}(U, V)$ so that the entry point is $D$-close to $\pi_S(U)$ and the exit point $D$-close to $\pi_S(V)$. 
\end{lemma}


\section{Rotating family and embedding free semigroups}\label{Section5}
In this section,  we apply the theory of rotating familly to the coned-off of quasi-tree of spaces $\mathbb S$ in the previous section. This allows to embed the free semigroup provided in Section \ref{Section3} into small cancellation quotients as in Lemma \ref{lastlem}. This completes the proof of Theorem \ref{mainthm}.

\subsection{Cone-off of the quasi-tree of spaces and rotating family}

We first introduce a construction of a   hyperbolic metric space by conning off a collection of quasi-convex subsets from a given hyperbolic sapce. 

Let $Z$ be a hyperbolic space with a collection $\mathbb S$ of uniformly quasiconvex subspaces. Assume that $\mathbb S$ has bounded intersection.  For $r\ge 0$, we first define the \textit{hyperbolic cone-off} $\dot Z_r( \mathbb S)$ of $Z$ along $\mathbb S$. 

For each $S\in \mathbb S$, the \textit{hyperbolic cone} $C_r(S)$ is the quotient space of the product $S\times [0,r]$ by collapsing $S\times 0$. The collasped point is called  the \textit{apex} o  denoted by $a(S)$, and $S\times 1$ the \textit{base} of the cone. The space is equiped with a geodesic metric so that it is the metric completion of the universal covering of a closed hyperbolic disk punctured at the origin.

The \textit{hyperbolic cone-off}  $\dot Z_r(\mathbb S)$ is the quotient space of the disjoint union $$Z\coprod_{S\in \mathbb S} C_r(\mathbb S)$$ by gluing $S$ with the base of the cone $C_r(S)$, equipped with length metric. Since $\mathbb S$ has bounded intersection, then for $r\gg0$,   $\dot{Z}_r(\mathbb S)$ is also hyperbolic \cite [Theorem 3.5.2 (E)]{1}.

Assume that $G$ acts isometrically on $Z$ and leaves invariant $\mathbb S$.  The action naturally extends by isometry to the  {hyperbolic cone} $C_r(S)$ by the rule 
$g(x,t)=(gx,t)$ for any $g\in G$, $x\in S$, $0\leq t \leq r$. This is a prototype of the notion of a rotating family introduced in \cite{2}.

\begin{dfn}
Assume $G$ acts isometrically on a metric space $\dot Z$. Let $A$ be a $G$-invariant set in $\dot Z$ and a collection of subgroups $\{G_a|a\in A\}$ of $G$ such that $G_a(a)=a$, $gG_ag^{-1}=G_{ga}$ for any $a\in A$, $g\in G$.

We call such a pair $(A,\{G_a:a\in A\})$ a \textit{rotating family}.
\end{dfn}

Returning to the above cone-off construction, the apexes $A=\{a(S): S\in \mathbb S\}$ and the stabilizers $G_a$ for $a\in A$ together consist of a rotating family.

From now on, we fix a contracting element $h$ which is independent from $f$ given in Lemma \ref{freesemigroup} and consider  the contracting system  $\mathbb S=\{g\ax(h): g\in G\}$ with   bounded interseciton. The existence of $h$ is assured by the non-elementary group $G$, since there are inifinitely many pairwise independent contracting elements under this assumption (cf. \cite[Lemma 2.12]{4}).

By Theorem \ref{quasitreeThm}, there exists $L>0$ such that the constructed quasi-tree $\QT$ of spaces is  a quasi-tree, on which $G$ acts coboundedly with a hyperbolic element $h$. We then apply the cone-off construction as above to $Z=\QT$ along $\mathbb S$. 

In view of projection complex, we have the following connection with the above hyperbolic cone-off space. 
\begin{lemma}
The projection complex $\PC$ is quasi-isometric to the hyperbolic cone-off $\dot Z_r(\mathbb S)$ of the quasi-tree of spaces $\QT$ along $\mathbb S$.
\end{lemma}

Despite of this fact, we still choose to work with the cone-off space since it allows us to exploit the theory of very rotating family developed by Dahmani-Guirardel-Osin \cite{2} in this setting. 

Roughly speaking, a rotating family $(A,\{G_a:a\in A\})$ is called \textit{very rotating} if every element in $G_a$ rotates around $a$ with a very big angle. This big angle is usually acheived by taking a sufficiently deep subgroup of $G_a$. This is the content of the following result. 
\begin{lemma}\label{RotFamilyLem}
There exist universal constants $\delta>0$, $r>20\delta$ such that for any hyperbolic element $h\in G \curvearrowright \QT$, there exist
$k=k(h)$, $l=l(h)>0$ with the following property.

Consider the cone-off space $ \dot Z_r(\mathbb S)$ over the scaled metric space $Z_l=(\QT,l\cdot d)$, where $d$ is the metric of $\QT$.  For every $n\geq 1$, set $$\mathbb E_n=\{g \langle h^{nk}\rangle g^{-1}: g\in G \}.$$

Then  $(A(\mathbb S),\mathbb E_n)$ is a very rotating family on the $\delta$-hyperbolic space $\dot{Z_l}(\mathbb S)$.
\end{lemma}
\begin{proof}[Sketch of proof]
Its proof is identical to that of \cite [Lemma 8.4]{3}, where the action of $G$ on $Z$ is assumed to be proper. In fact, the properness is required only when verifying that the system $\mathbb S$ has bounded intersection  and the action of $E(h)$ acts properly on $\ax(h)$. In the current setting, these two facts are guaranteed by Lemma \ref{RotonQT}. Thus, we have the same result.
\end{proof}


Modulo Lemma \ref{RotFamilyLem}, we obtain the following result by the same proof of \cite [Lemma 8.8]{3} using the quantitative Greendlinger's Lemma in \cite{2}.  
\begin{lemma}\label{deepPtLem}
There exist constants $\epsilon=\epsilon(h)$, $k=k(h)>0$ with the following property. For any $n\ge 1$, there exists $M=M(h,n)>0$ for $n>1$ with $M(h,n)\rightarrow \infty$ as $n\rightarrow \infty$ such that for any $1\neq g\in  \langle\langle h^{kn}\rangle\rangle$, any geodesic $[o,go]$ in $\QT$ contains a $(\epsilon,M)$-deep point in some $S\in \mathbb S$.
\end{lemma}

\subsection{Embedding free semigroups into quotients: end of the proof of Theorem  \ref{mainthm}}\label{SSection52}

The following result is the last step in proving Theorem \ref{mainthm}.

\begin{lemma}\label{lastlem}
For each $0<\delta<\delta_G$, there exist  a free semi-group $\Gamma$   and  an integer $n\ge 1$ such that the map $$X\rightarrow X/{\langle\langle h^{kn}\rangle\rangle}$$ is injective on $\Gamma o$ and $\delta_\Gamma\ge \delta$.

\end{lemma}

Assuming Lemma \ref{lastlem}, we complete the proof of Theorem \ref{mainthm} as follows.  Note  that  the quotient metric $\bar{d}$ is smaller than the metric $d$, so we have $$\delta_{G/\langle\langle h^{nk}\rangle\rangle}
\geq \delta_{\Gamma}\geq \delta.$$
Thus, we get $\delta_{G/\langle\langle h^{nk}\rangle\rangle}
\rightarrow \delta_G$ concluding the proof.

We now give the proof of Lemma \ref{lastlem}.
\begin{proof}[Proof of Lemma \ref{lastlem}]
 By Lemma \ref{freesemigroup},   there exist a free semi-group $\Gamma$ and an integer $n_\delta>0$ such that the standard Cayley graph $\mathcal T$ is mapped naturally to a quasi-geodesic tree rooted at $o\in X$. In fact,  each geodesic branch originating from $1$ in $\mathcal T$ labels a $(n_\delta,\theta)$-admissible path in $X$ relative to the contracting system $\mathbb S=\{g\ax(h): g\in G\}$.
 
Let $n\gg n_\delta$ be a big integer  chosen later (see the Claim \ref{deepinalpha2}). We now proceed by contradiction. Assume that there exist two elements $g_1, g_2\in \Gamma$ such that $$g_1 \langle\langle h^{kn}\rangle\rangle \cdot o=g_2 \langle\langle h^{kn}\rangle\rangle \cdot o.$$ Without loss of generality, assume that there exists $1\neq g \in \langle\langle h^{kn}\rangle\rangle$ such that $g_1o =g_2 go$ for the action   $G \curvearrowright X$. 

\textbf{Step 1.}
We first consider the action of $G$ on  the quasi-tree of spaces $\QT$. Since  $S=g\ax(h) \in \mathbb S$ and $\ax(h)$ are geodesically embedded in $\QT$, we shall also understand the points $o$ and $go$ as  in the space $\QT$. Consider the   triangle given by  the following geodesics in $ \QT$,
$$\alpha=[o,g_1o]_{\mathcal C}, \beta=[o, g_2o]_{\mathcal C}, \gamma=g_2[o, go]_{\mathcal C},$$
where the subscript $\mathcal C$ is used to mean a geodesic in the $\QT$.

By Lemma \ref{deepPtLem}, for $g\in  \langle\langle h^{kn}\rangle\rangle$, there is $S_0\in \mathbb S$ such  that $\gamma$ contains a $(\epsilon, M)$-deep point in $S_0$, where $M=M(h, n)\to \infty$ as $n\to \infty$.
\begin{clm}\label{deepinalpha}
There exist $\tilde \epsilon=\tilde \epsilon(h), \tilde M=\tilde M(h, n)>0$ such that at least one of $\alpha, \beta$ has a $(\tilde\epsilon,\tilde{M})$-deep point in  $S_0$, where $\tilde{M}\to \infty$ as $n\to\infty$.
\end{clm}

\begin{proof}[Proof of the claim] 
By Theorem \ref{quasitreeThm}, we know that $\QT$ is a hyperbolic space. The conclusion is easy consequence of thin-triangle property together with quasiconvexity of $S\in \mathbb S$. By computation, $\tilde M=\tilde M(h, n)$ is a definite positive fraction of $M$ and thus has the same asymptotic as $M$.  We leave the details to the interested reader. 
\end{proof}

By Claim \ref{deepinalpha}, assume for definiteness that the geodesic $\alpha=[o, g_1o]_{\mathcal C}$ in $\QT$ has a $(\tilde \epsilon, \tilde M)$-deep point in $S_0$.  

\textbf{Step 2.} We now turn to consider the action of $G$ on $X$. Look at the intervel set $\mathbb S_{ K'}(U, V)$ between $U:=\ax(h)$ and $V:=g_1\ax(h)$ for a constant $K' $ satisfying Lemma \ref{unicloseLem}. By Lemma \ref{admisIntvLem}, there exist a constant $K_0=K_0(K')>0$  and   a $(  K_0, \theta)$-admissible path $\tilde \alpha$  between $o\in U$ and  $g_1o\in V$ with the  saturation $\mathbb S_{ K'}(U, V)$. Notice  that this admissible path $\tilde \alpha$ naturally projects to a $K'$-standard path $\bar \alpha$ with endpoints $o$ and  $g_1 o$: it is the path connecting $o$ and $g_1o$ and passing exactly each contracting subset in $\mathbb S_{K'}(U, V)$ by a geodesic. See Defintion \ref{stdpathDef}.


  It is possible that $\mathbb S_{K'}(U, V)=\emptyset$: $\tilde \alpha$ is a geodesic in $X$ between $o$ and $g_1o$, and correspondingly, $\bar \alpha$ is just an edge of length $L$ in $\QT$. However, we shall show that it must contain $S_0$ for big $n\gg 0$.

\begin{clm}\label{deepinalpha2}
There exist $n_0\ge 0$ and $\tilde K=\tilde K(h,n)>0$ for $n\ge n_0$ such that $S_0 \in \mathbb S_{ K'}(U, V)$ and  $d^{\pi}_{S_0}(U, V)\ge \tilde K$. Moreover, $\tilde{K}\to \infty$ as $n\to\infty$.
\end{clm}
\begin{proof}[Proof of the claim] 
By Lemma \ref{unicloseLem}, the proof consists in comparing the standard path $\alpha$ and the geodesic $\bar \alpha$ of the same endpoints in the quasi-tree of spaces $\QT$. 
Namely, there is a uniform constant $D>0$ such that every contracting set   in $\mathbb S_{K'}(U, V)$ has the   $D$-neighborhood intersecting  $\alpha$.

Since  $\mathbb S$ has bounded intersection in $\QT$ by Lemma \ref{RotonQT},  there exists $L_0=\mathcal R(\max(D, \epsilon))>0$ so that 
\begin{equation}\label{bddintEQ}
\diam{N_D(S)\cap N_\epsilon(S')}\le L_0
\end{equation}
for any two $S\ne S'\in \mathbb S$, where the diameter is measured using the metric in $\QT$.

Recall that $d_S^{\pi}(U, V)\ge K'$ for each $S\in \mathbb S_{K'}(U, V)$. Since each $S$ is convex in a hyperbolic space $\QT$, it is easy to see that  by taking  $K'\gg 2D+L_0$, we obtain 
$$
\diam{N_D(S)\cap \alpha}\ge L_0.
$$

By the Claim \ref{deepinalpha}, $\tilde M=\tilde M(h, n)\to \infty$ as $n\to \infty$. By definition of $(\epsilon, M)$-deep points, we can take $n$ big enough so that  
\begin{equation}\label{bddintEQ2}
\diam{N_\epsilon(S_0)\cap \alpha}\ge \tilde M/ 2>L+L_0+
2\tilde \epsilon.
\end{equation}

Assume by contradiction that  $S_0\notin \mathbb S_{K'}(U, V)$. By (\ref{bddintEQ}),  the $\epsilon$-neighborhood of $S_0$ intersects boundedly by a constant $L_0$ with the $D$-neighborhood of each $S\in \mathbb S_{K'}(U, V)$. Since $\diam{N_\epsilon(S_0)\cap \alpha}>L_0$ by (\ref{bddintEQ2}), $N_\epsilon(S_0)\cap \alpha$ has to lie between the two intersections with $\alpha$ of adjacent contracting sets  $S_1, S_2\in \mathbb S_{K'}(U, V)$.  By definition of standard path, there are only   edges of length $L$ between $S_1$ and $S_2$.  This thus gives a contradiction with (\ref{bddintEQ2}). Hence, it is proved that  $S_0\notin \mathbb S_{K'}(U, V)$.
 \end{proof}

In conclusion, we have proven by the Claim \ref{deepinalpha2}, that the interval set $\mathbb S_{K'}(U, V)$ contains $S_0$ provided by the Claim \ref{deepinalpha} so that  $d^{\pi}_{S_0}(U, V)\ge \tilde K$. By Lemma \ref{pp} (4), we have 
$$
\diam{N_C(S_0) \cap [o, g_1o]_X} \ge \tilde K/2, 
$$
where $[o, g_1o]_X$ is a geodesic in $X$.  On the other hand, for $g_1\in \Gamma$, there exists a $(n_\delta, \theta)$-admissible path $\gamma$ between  $o$ and $g_1o$. Using now the bounded intersection of $\mathbb S$ in $X$,  since $\tilde K\to \infty$ as $n\to \infty$, we get a contraction for $n\gg n_\delta$. The proof of the lemma is thus  complete.
\end{proof}

\section{Products of group actions}\label{Section6}
In this section, we shall give a proof of  Theorem \ref{productThm} which generalizes Theorem \ref{mainthm} to a product of geometric actions with contracting elements.

Consider a product of non-elementary groups $G=G_1\times G_2\cdots\times G_n$ and a product of   geodesic metric spaces $X=(X_1,d_1)\times (X_2,d_2)\cdots\times (X_n,d_n)$. Let $G\curvearrowright  X$ be the diagonal action  so that $G_i \curvearrowright (X_i,d_i)$ is a proper and co-compact action for  each $i=1,\cdots, n$.  We equip $X$ with $L^p$-metric for $1\leq p< \infty$ as follows
$$
d^p(\bar x, \bar y)=\left(\sum_{i=1}^{i=n} d(x_i, y_i)^p\right)^{1/p},
$$ where $\bar x=(x_i), \bar y=(y_i)$. For $p=\infty$, define $$d^\infty(\bar x, \bar y)=\max_{i} \{d(x_i, y_i)\}.$$
Then $G$ admits a proper and co-compact action by isometries on $X$ with the $L^p$-metric for $1\le p \le \infty$. Moreover, choosing  basepoints $o_i\in X_i$ in each factor, we can define the growth rate of a subset $A \cdot \bar o$ in $G$ where $\bar o=(o_i) \in X$.

\subsection{Proof of Theorem \ref{productThm}}
Recall that a quotient $\bar G$ of a product group $G=G_1\times G_2\cdots\times G_n$ is called \textit{proper} if the kernel of $G\to \bar G$ projects to to infinite normal subgroup in each factor $G_i$.  

The proof of the theorem follows the similar line as that of Theorem \ref{mainthm}, with the help of the additional ingredient from \cite [Proposition 5.1]{5}. It relates the growth rates of each $G_i\curvearrowright X_i$ with that of $G\curvearrowright X$ in a nice way.  

Denote  $N_i(r)=\sharp B(o_i, r)\cap A_io_i$ for a subset $A_i\subset G_i$. Denote by $\delta_i=\delta_{A_i}$ the growth rate of $A_i$ with respect to $G_i\curvearrowright X_i$.

\begin{prop}\label{pqnormLem}
Assume there exists a subset $A_i\subset G_i$ for each $i$ such that up to bounded error, $\log N_i(r)$ is subadditive in $r$.  For any $1\leq p\leq \infty,$ the growth rate $\delta_A$ of $A=\Pi^n_iA_i$ with respect to the $L^p$-metric on $X$ is the $L^q$-norm of $(\delta_1,\cdots,\delta_n)$, where $1/p+1/q=1$.
\end{prop}
\begin{rmk}
This subbaddtivity $\log N_i(r)$  is exactly the purely exponential growth in \cite{4}. One particular example with subadditive $\log N_i(r)$ is given by the growth function of a geometric action with a contracting element. 
\end{rmk}

To use Proposition  \ref{pqnormLem}, we let $A_i \subset G_i$ be the free semigroup given by Lemma \ref{freesemigroup} whose growth rate $\delta_i$ tends to $\delta_{G_i}$. Since $N_i(r) \asymp \exp(\delta_{G_i} r)$, we have $\log N_i(r)$ is subadditive in $r$, up to bounded error.

Thus, by Proposition \ref{pqnormLem}, the set $A=\Pi^n_iA_i$ with respect to the $L^p$-metric 
has the growth rate the $L^q$-norm of $(\delta_1,\cdots,\delta_n)$. Hence, as $\delta_i\to \delta_{G_i}$, we have $\delta_A$ tends to the $L^q$-norm of $(\delta_{G_1},\cdots,\delta_{G_n})$. 

Since $G_i\curvearrowright X_i$ is geometric with a contracting element, the growth function is purely exponential so the set $A_i=G_i$ satisfies Proposition \ref{pqnormLem} so the growth rate $\delta_G$ of $G$ using $L_p$-metric is the  $L^q$-norm of $(\delta_{G_1},\cdots,\delta_{G_n})$. We refer the reader to \cite{4} for the proof of purely exponential growth.

To complete the proof of theorem \ref{productThm}, we can choose for each $A_i$, a sufficiently high power $h_i^{m_i}$ of a contracting element $h_i$ so that $A_i o$ injects into the quotient space $X_i/\llangle h_i^{m_i}\rrangle$ acted upon by the quotient group action   $G_i/\llangle h_i^{m_i}\rrangle$. This is the same as   the proof of Lemma  \ref{lastlem}.

The injectiveness implies that  the growth rate of the group $\prod_{i=1}^{n} G_i/\llangle h_i^{m_i}\rrangle$ on $\prod_{i=1}^{n} X_i/\llangle h_i^{m_i}\rrangle$ is greater than the growth rate of $A$ with respect to the  action $G\curvearrowright X$ which is the $L^q$-norm of $(\delta_1,\cdots,\delta_n)$.

Denote $$\bar G_m=G/{\prod^n_{i=1}\llangle h_i^{m_i}\rrangle}={\prod^n_{i=1}G_i/\llangle h_i^{m_i}\rrangle}$$
 we obtain
$$
\delta_{\bar G_m}\rightarrow\delta_G, \:as \: m_i\rightarrow \infty.
$$
The proof of Theorem \ref{productThm} is complete.

\subsection{Applications to cubical groups: proof of Corollary \ref{CubeNoGap}}
 A group $G$ is called \textit{cubical} if it admits a geometric action on a CAT(0) cube complex. We now explain a proof of Corollary \ref{CubeNoGap} for cubical groups which drops the assumption of rank-1 elements in Corollary \ref{CAT0NoGap}.

By a theorem of  Caprace and Sageev  \cite[Corollary 6.4]{12}, there exists a finite index subgroup $\dot G$ of $G$ acting on a convex subcomplex in $X$, which splits as a product of action on a product of  cubical subcomplexes such that each factor contains a rank-1 element. Without loss of generality, we can assume that $\dot G$ is a normal subgroup.

By Theorem \ref{productThm}, $\dot G$ posseses a sequence of proper quotients $\dot G/N_m$ with growth rates tending to $\delta_{\dot G}$, where $N_m$ is an infinite normal subgroup of $\dot G$. We denote by $\hat N_m$ the normal closure of $N_m$ in $G$. Then $N_m=\dot G\cap \hat N_m$. It thus follows that the homomorphism $$\dot G/N_m\to G/\hat N_m$$ defined by $\dot gN_m\mapsto \dot g\hat N_m$   is injective.

We shall show that the growth rate of $G/\hat N_m$ tends to $\delta_G$.  Since $\delta_{\dot G}=\delta_{G}$ holds finite index subgroups, it suffices to show that $\delta_{G/\hat  N_m}$  is greater than $\delta_{\dot G/N_m}$. 

To this end, we follow the proof of the convergence $\delta_{\dot G/N_m} \to \delta_{\dot G}$ by embedding a large free semigroup $\Gamma\subset \dot G\subset G$ into $\dot G/N_m$ with $\delta_\Gamma \to \delta_{\dot G}$ (cf. Lemma \ref{lastlem}). From the injectiveness of the above homomorphism, we obtain that $\Gamma$ injects to $G/\hat N_m$ as well and thus $\delta_{G/\hat N_m}\ge \delta_{\Gamma}$. This implies that  $\delta_{G/\hat N_m} \to \delta_{\dot G}$ and the corollary  \ref{CubeNoGap} follows.

\end{document}